\documentclass{article}
\usepackage[utf8]{inputenc}

\newcounter{rmk}

\usepackage{amsmath}
\usepackage{float}

\usepackage{comment}
\usepackage[linktocpage=true]{hyperref}
\usepackage[margin=1.2 in, top=1 in, bottom= 1.2 in]{geometry}

\usepackage{indentfirst} 
\usepackage[english]{babel}
\usepackage{amsfonts}
\usepackage{mathrsfs}
\usepackage[mathscr]{euscript}
\usepackage{bbm}
\usepackage{latexsym}
\usepackage{mathdots}
\usepackage{amssymb}
\usepackage{mathtools}
\usepackage{graphicx}
\usepackage{tikz}
\usepackage{placeins}
\usepackage{esint}

\usepackage{enumitem}
\setlist{nolistsep}
\usepackage{amsthm}
\usepackage{relsize}
\usepackage{nicefrac}
\usepackage[capitalize]{cleveref}

\newtheoremstyle{plain}{3mm}{3mm}{\slshape}{}{\bfseries}{.}{.5em}{}
\newtheoremstyle{definition}{2mm}{2mm}{}{}{\bfseries}{.}{.5em}{}
\theoremstyle{plain}
	
\newtheorem{theorem}{Theorem}

\newtheorem{corollary}[theorem]{Corollary}

\theoremstyle{definition}

\newtheorem{remark}[rmk]{Remark}

\theoremstyle{plain}
\newtheorem*{namedthm}{\namedthmname}
\newcounter{namedthm}
	

\newcommand{\R}{\mathbb{R}}

\newcommand{\eps}{\epsilon}

\newcommand{\B}{\mathcal{B}}
\newcommand{\G}{\mathcal{G}}

\title{A counterexample for pointwise upper bounds on Green's function with a singular drift at boundary.}
\author{Aritro Pathak}

\date{}

\begin{document}

\maketitle
\begin{abstract}

    We show an example of a sequence of elliptic operators in the unit ball with drifts that diverge as the inverse distance to the boundary, for which we don't get uniform upper estimates for the Green's function with the pole at the origin. Such drifts have been considered in the literature in the study of the $L^{p}$ Dirichlet problem for both the parabolic and elliptic operators. Our construction provides a counterexample to an earlier claim of Hofmann–Lewis.
\end{abstract}

\section{Introduction.}

Consider the unit ball of radius $1$ with center at origin, $\Omega=B(0,1)\subset \R^n$, $n\geq 3$, and the sequence of operators, for $m\geq 100$\footnote{The number 100 is chosen large enough for convenience so that the argument of \cref{thm1} remains unaltered for the special case of $\G_m(x,0)$ in the ball $B(0,1)$, and the test functions considered in that proof are disjoint from the region $|r|>1- \frac{1}{m}$, for all $m\geq 100$.},
\begin{align}\label{one}
L_m u=-\nabla\cdot(\nabla u)+\B_m \cdot \nabla u
\end{align}    
where for each $m$,  $\B_m$ is the drift which is of the form,
\begin{align}\label{two}  \B_m=\begin{cases}
    -\frac{C}{1-|r|}\hat{r}& |r|\leq 1-\frac{1}{m} \\
    -Cm\hat{r}& 1-\frac{1}{m}<|r|<1, 
\end{cases}\end{align}
for any arbitrary constant $C\geq 1$. Here, we use the notation that $\hat{r}$ is the unit vector at the point $\hat{r}=\frac{\vec{r}}{|\vec{r}|}$, with $|\vec{r}|$ the magnitude of the vector $\vec{r}$.

As $m\to \infty$, this approximates a drift $\B(x)$ that behaves like the inverse distance to the boundary, that has been considered in more general domains $\Omega$, in \cite{HL01,KP01};
\begin{align}\label{imp2}
     \delta(x) |\B(x)| <C.
\end{align}
Here, $\delta(x)$ denotes the distance of the point $X$ to the boundary.

One is referred to recent works on the existence of the Green's function in the distributional sense, pointwise bounds on Dirichlet Green's functions in certain settings, and some scale invariant regularity estimates of solutions in the setting of Lorenz spaces and the Kato spaces,  in \cite{KS19,Sak21a,Sak21b,Mou23}.

The following existence and uniqueness (up to a set of measure zero) result for the Green's function is well known, where the drift is bounded, in bounded domains. See for example, Section 5 of  \cite{KS19}. or more generally Theorem 6.1 of \cite{Mou23}.

\begin{theorem}\label{thm1}
    For any bounded domain $\Omega\subset \R^{n}$, there exists a unique non-negative function $\G_m:\Omega\times \Omega\to \R\cup \{\infty\}$, called the Green's function associated with $L_m$, such that the following holds:
    \begin{enumerate}
    \item $G_m(\cdot,y)\in\mathcal{C}(\Omega\setminus B(y,s))\cap W^{1,2}(\Omega\setminus B(y,s))\cap W^{1,1}(\Omega)\ \ \ \forall y\in \Omega, \forall s>0$.
    \item $\int_{\Omega} \Big(\nabla_{x}\G_m(x,y)\cdot \nabla \phi(x)+\B_m \cdot(\nabla_{x}\G_m(x,y))\phi(x)\Big) dx =\phi(y)$ for all $\phi\in C_{c}^{\infty}(\Omega).$
\end{enumerate}
\end{theorem}

For the case of an operator with only the elliptic principal term and no lower order terms in a domain $\Omega$, pointwise upper and lower estimates for the Green's function with the pole at $x\in \Omega$ in the interior region $B(x,\delta(x)/2)=\{y\in \Omega||x-y|\leq \delta(x)/2\}$  were obtained in \cite{GW82}. Precisely, they showed the existence of constants $C_1,C_2$ uniform on the domain, so that for any $y\in $ $B(x,\delta(x)/2)$, we have 
\begin{align}\label{pointwise}
    \frac{C_1}{|x-y|^{n-2}}\leq \G(x,y) \leq  \frac{C_2}{|x-y|^{n-2}}.
\end{align}

For the case of a drift that is bounded from above by $\frac{\eps}{\delta(X)}$ where $\eps$ is sufficiently small, we also show pointwise upper and lower bounds on the Green's function in the \cite{Pat25a}. In fact, the argument for the lower bound in this paper is essentially the same as in \cite{Pat25a}. It will be apparent in Section 3 of this paper, that when we choose $C$ in \cref{two} small enough, the counterexample fails.

We show the following in Section 3:
\begin{theorem}
    For the operator considered in \cref{two}, for any $m\geq 100$, with the drift considered in \cref{two}, with $C=1$, the Dirichlet Green's function, defined in the distributional sense, $\G_m(x,0)$, when evaluated at the point $(\frac{1}{2},0,\dots,0)$  diverges to infinity as $m\to \infty$ .
\end{theorem}

We note that the claimed proof in Theorem 4.3(a) in Chapter III of \cite{HL01} on the upper bound on the Green's function for the elliptic operator in the half plane in $\R^{n}$, with $n\geq 3$, with a singular drift term bounded by the inverse distance to the boundary, does not work as intended, as the purported elliptic version of Lemma 2.10 in Chapter III of \cite{HL01} can't be proved, since in turn it relies on the upper bound on the Green's function for the parabolic equation with the drift term which is proved in Chapter I of \cite{HL01} and which does not generalize to the elliptic case. 

In another direction, when we consider a drift inside the unit ball, that is slightly less singular at the boundary, and which diverge as $1/(1-|r|)^{1-\beta}$ for some $0<\beta<1$, then solutions exist and one gets expected sharp pointwise estimates on the Green's function. See for example, \cite{Pat25b, Ha24}. 

We first show that lower bounds can be obtained for the Green's function, uniform in $m$, for the Green's function in the domain $B(0,\frac{1}{2})$. 

The Green's function $\G_m(x,0)$ in this special case is dependent only on the radial variable, and thus $\G_m(x,0)=\G_m(|x|,0)$. 

\begin{theorem}
    The Green's function as defined, $\G(x,0)$ is dependent only on the radial variable, and thus we have, $\G(x,0)=\G(|x|,0)$ for any $x\in B(0,1)\setminus\{0\}$.
\end{theorem}
\begin{proof}
    We have the equation 
    \begin{align}\label{eqq}
    L_m u_{m,\rho} =\frac{1}{|B(0,\rho)|}1_{B(0,\rho)},
    \end{align}
    for any $\rho<\frac{1}{2}$, where we write $x=(r,\theta_1,\dots,\theta_{n-1})$. For convenience, we suppress the dependence on $m$ in the notation. Fix any $\alpha\in \mathbb{S}^{n-1}$, and the co-dimension one hyperplane $H_\alpha$ perpendicular to $\alpha$ and passing through the origin. One can orient the axes to define the variable $\theta_{n-1}$ as the azimuthal angle on $H_\alpha$. 

    We write, $x'(x)=(r',\theta'_1,\dots,\theta'_{n-1})=(r,\theta_1,\dots,\theta_{n-1}+\beta)$.
    
    Then consider the function $u_{\rho,\beta}(x)=u_{\rho}(x')=u_\rho(r,\theta_1,\dots, \theta_{n-1}+\beta)$, for some fixed $0\leq\beta<2\pi$. For now, we suppress the dependence of $u_{\rho,\beta}$ on $\alpha$ for notational convenience. We note that this is a well defined function in the unit ball, and we show that this also satisfies \cref{eqq}.  

    We have, 
    \begin{align}
        \frac{\partial }{\partial r}(u_{\rho,\beta}(x))=\frac{\partial }{\partial r}(u_{\rho}(x'))= \frac
       {\partial u_{\rho}}{\partial r'}(x')\frac{\partial r'}{\partial r}+\sum\limits_{i=1}^{n-1}\frac{\partial u_{\rho}}{\partial \theta'_i}(x')\frac{\partial \theta'_i}{\partial r}=\frac{\partial }{\partial r'}(u_{\rho}(x')).
    \end{align}
   and we have $\partial r'/\partial r =1$ and for each $i\in \{ 1,\dots, n-1\}$, that $\partial \theta'_i /\partial r=0$.

   By a similar argument, we also have for any $1\leq i\leq (n-1)$, 
   \begin{align}
        \frac{\partial }{\partial \theta_i}(u_{\rho,\beta}(x))=\frac{\partial }{\partial \theta'_i}(u_{\rho}(x')).
   \end{align}
    
    Using the structure of the Laplacian we get, noting also that $\B_m(x), \B_m(x')$ both have only a radially inward pointing component at each of these points, of equal magnitude,
    \begin{multline}
        \nabla\cdot(\nabla u_{\rho,\beta}(x)) +\B_{m}(x)\cdot \nabla u_{\rho,\beta}(x)= \nabla'\cdot(\nabla' u_{\rho}(x') +\B_{m}(x')\cdot \nabla' u_{\rho}(x') \\= \frac{1}{|B(0,\rho)|}1_{B(0,\rho)}(x')= \frac{1}{|B(0,\rho)|}1_{B(0,\rho)}(x),
    \end{multline}
where the last equality follows from the radial symmetry of the function, and we write, $\nabla':=\big(\frac{\partial}{\partial r'},\frac{\partial}{\partial \theta'_1},\dots, \frac{\partial}{\partial \theta'_{n-1}}\big)$.
    
    This gives us a family of solutions $u_{\rho,\beta}$ for $0\leq \beta< 2\pi$. By uniqueness of the solutions, we conclude that all these functions are identical. In particular this forces that for any $r\leq 1$, the value of $u_{\rho,\beta}$ on $S(0,r)\cap H_\alpha$, where $S(0,r)$ is the sphere of radius $r$ centered at the origin, is constant.

    Now, consider the family of unit vectors $\alpha\in \mathbb{S}^{n-1}$, and repeat this argument for any arbitrary $\alpha$, rotating the coordinates to have $\theta_{n-1}$ to be the azimuthal angle on $H_{\alpha}$. 

    We thus have a family of solutions, all of which coincide., and by repeating the previous argument, this forces $u_\rho(x)=u_{\rho}(|x|)$.

    Now by a standard limiting argument, see, for example Section 6 of \cite{Mou23}, we eventually get pointwise convergence almost everywhere of $u_{m,\rho_i}\to u_m$ in $B(0,1)\setminus \{0\}$ along a subsequence $\rho_i \to 0$. We are thus forced to conclude that $u_m(x)=u_m(|x|)$.

\end{proof}

\section{Lower bounds on the Green's function.}

The lower bound follows by an argument similar to that used in \cite{GW82}.\footnote{This was also outlined in the argument of Lemma 4.3 in Section III in \cite{HL01}.}

  We consider the drift more generally considered in \cref{imp2}, and the argument remains the same in the special case of the Green's function $\G(x,0)$ for the operator in \cref{one}, for the drift considered in \cref{two}. In particular in this case, the constant $c_0$ in \cref{thm1} below is independent of the level of truncation $m$ in the drift of \cref{two}.

We remark that the Harnack inequality,  used routinely in the proof of \cref{thm1} below, can be used with constants independent of the truncation level $m$, and can also be used for the limiting drift of \cref{imp2}, as long as the balls considered stay uniformly bounded away from the boundary. That is the case in the setting of \cref{thm1}. One notes the form of the Harnack constant in \cite{GT01} in Chapter 8 for example, which shows that for any ball $B(x,\delta(x))\eta)$, for uniform $\eta$, the Harnack constant is bounded from above by $ C_{0}^{\eta}$. As we get arbitrarily close to the boundary, as long as we consider balls bounded away from the boundary in this manner, the Harnack inequality can thus be used routinely.


\begin{theorem}\label{thm1}
    For the elliptic operator with the coefficients satisfying \cref{two} in $B(0,1)\subset \R^n$, we have the lower bound for the Green's function for the operator in \cref{one}: for any $z,y\in B(0,1)$ with $|z-y|\leq \frac{1}{2}\delta(y):=\frac{1}{2}\text{dist}(y,\partial B(0,1))$: 
    \begin{equation}\label{lower}
        \G(y,z)\geq c_0\frac{1}{|z-y|^{n-2}}.
    \end{equation}
\end{theorem}

\begin{proof}[Proof of \cref{thm1}]

    The proof essentially follows by extending the argument of the proof of Eq.(1.9) of \cite{GW82}. Take $r:=|z-y|$. Consider a smooth cut-off function $\eta$ which is $1$ on $B(y,r)\setminus B(y,r/2)$ and zero outside $B(y, 3r/2)\setminus B(y,r/4)$, and further $0\leq \eta\leq 1$ and $|\nabla \eta| \leq \frac{K}{r}$.

    Henceforth, we use the Einstein summation convention, where the summation sign is implied.

    Given the domain $B(0,1)$, for any admissible test function $\phi$, the Green's function satisfies the following adjoint equation;
    \begin{equation}\label{eqimpp}
        \int_{B(0,1)} \Big((\nabla \phi) \cdot\nabla \G(y,x) +\G(y,x) \B\cdot \nabla \phi \Big)dx =\phi(y)
    \end{equation}

    We consider the test function $\phi(x)= G(y,x)\eta^2(x)$, and first get,
    \begin{multline}
        \int_{B(0,1)} |\nabla \G(y,x)|^2 \eta^2 dx=  -\int_{B(0,1)}\Big(2\eta \G(y,x)\nabla \G(y,x)\cdot \nabla \eta  +\G(y,x)\eta^2 \B\cdot\nabla \G(y,x) \\+2\G^2(y,x)\eta \B\cdot \nabla \eta)\Big) dx.
    \end{multline}

    Now by using the bound on the drift term $\B$, the Cauchy inequality with $\epsilon$'s, with small enough $\eps$, for the first two terms on the right,
    \begin{multline}
        \int_{B(0,1)} \G(y,x)\eta \nabla\G(y,x)\cdot \nabla \eta dx\leq \eps\int_{B(0,1)} \eta^2 |\nabla\G(x,y)|^2 dx +\frac{K^2}{\eps r^2}\int_{r/4 < |x-y|< 3r/2} \G(y,x)^2 dx,\\
        \int_{B(0,1)}\G(y,x)\eta^2\nabla\G(y,x)\cdot\B dx\leq  \eps\int_{B(0,1)} \eta^2 |\nabla\G(x,y)|^2 dx +\frac{\tilde{C}^2}{\eps \delta(y)^2}\int_{r/4< |x-y|<3r/2} \G(y,x)^2  dx.
    \end{multline}

    Here, the constant $K$ is due to the bound on the gradient of $\eta$ and the constant $\tilde{C}$ is due to the bound on the $|\B|$ term uniformly within the annulus where $r/2\leq |x-y|<3r/2$ .  Using the bounds on the cut-off function $\eta$ introduced above, and hiding the term with the square of the gradient of $\G(y,x)$, we get, 
   \begin{multline}\label{eq13}
        \int\limits_{r/2<|x-y|<r} |\nabla \G(y,x)|^{2} dx \leq \Big(K_1 \frac{1}{r^{2}}\cdot \int\limits_{r/4<|x-y|<3r/2} \G(y,x)^{2}dx\Big) +\Big(K_2 \frac{1}{r\delta(y)}\cdot\int\limits_{r/4<|x-y|<3r/2} \G(y,x)^{2} dx\Big) \\ +\Big(K_3 \frac{1}{\delta(y)^2}\cdot\int\limits_{r/4<|x-y|<3r/2} \G(y,x)^{2} dx\Big) 
    \end{multline}

    Here, the constants $K_1, K_2, K_3$ are obtained from $\tilde{C}, K, \epsilon,$ after hiding the term with the square of the gradient of $\G(y,x)$ above. Noting that $r\leq \frac{1}{2}\delta(y)$, we get with a further altered constant $\tilde{K}$, 
       \begin{align}\label{eq6}
        \int\limits_{r/2<|x-y|<r} |\nabla \G(y,x)|^{2} dx \leq \tilde{K} \frac{1}{r^{2}}\cdot \Big(\int\limits_{r/4<|x-y|<3r/2} \G(y,x)^{2}dx\Big)\leq \tilde{K}r^{n-2}\big(  \sup\limits_{r/4<|x-y|< 3r/2} \G(y,x)^{2} \big) .
    \end{align}
    Again as in \cite{GW82}, choose a similar cut-off function $\phi$ that is 1 on $B_{r/2}(y)$ and zero outside $B_{r}(y)$, and using it as the test function we get,
    \begin{multline}
     1= \int\limits_{r/2< |x-y|< r} (\partial_{i}\G(y,x) \partial_{i} \phi + \G(y,x) \B_i \partial_i \phi \Big) dx)\leq  \frac{K}{r} \int\limits_{r/2< |x-y|< r} |\nabla \G(y,x)|dx  \\+\frac{\tilde{C}}{r\delta(y)}\int\limits_{r/2< |x-y|< r} \G(y,x) dx.
    \end{multline}


    Using the identity of \cref{eq6}, and Cauchy's inequality for the first term on the right, along with a trivial volume bound, and finally Harnack's inequality, with some constants $c_0,c_1$ we finally get,
\begin{align}
     1\leq \frac{1}{c_1} r^{n-2} \sup\limits_{r/4< |x-y|< 3r/2} \G(y,x)
     \leq \frac{1}{c_0} |z-y|^{n-2}\G(y,z).
\end{align}

     \end{proof}

     Using Harnack inequality, and \cref{thm1}, we immediately get,
     \begin{corollary}\label{cor1}
         For the elliptic operator with the coefficients satisfying \cref{two} in $B(0,1)\subset \R^n$, we have the lower bound for the Green's function for the operator in \cref{one}: for any $z,y\in B(0,1)$ with $|y-z|\leq \frac{1}{2}\delta(z)=\frac{1}{2}\text{dist}(z,\partial B(0,1))$:
    \begin{equation}\label{lower}
        \G(y,z)\geq c_0\frac{1}{|z-y|^{n-2}}.
    \end{equation}
     \end{corollary}

     In other words, we can interchange the order of the arguments for the Green's function, in this result. 
     
\section{Counterexample for uniform upper bounds.}

\begin{proof}[Proof of Theorem 2] In this case, we consider the domain $B(0,1)$ and for all $m\geq 3$, the operator $L_m$ with the drift $\B_m$ as considered in \cref{two} in the introduction.

For the points on the $x$-axis, henceforth we simply write $a$ in place of the point $(a,0,\dots,0)$. We consider the above ODE of \cref{ode} along the $x$-axis.

Consider the pole of the Green's function to be at the origin. Since the Green's function $\G(x)=\G(x,0)$ in this case is dependent only on the radial coordinate, after writing the expression of \cref{one} in polar coordinates, noting that the generalized solution of \cref{thm1} is in this case also a classical solution to $L_m u=0$ in $B(0,1)\setminus \{0\}$, we will actually solve in $(0,1)\setminus \{0\}$ the equation,
\begin{align}\label{ode}
    \frac{d\G_{m}^{2}}{d r^2} +\frac{n-1}{r}\frac{d\G_m}{d r} +\B_m \cdot \hat{r}\frac{d\G_m}{d r} =0.
\end{align}

The level sets of the Green's function are the spheres with center at the origin, and $\frac{d\G}{dr}<0$ for all $0<r<1$. By definition, $\G(t)=0$ when $|t|=1$. In this construction, we will later use the fact that the lower bound from \cref{thm1} exists with point-wise bounds independent of $m$, in the subdomain $B(0,\frac{1}{4})$.

By using a standard integrating factor, we get from \cref{ode} that, 
\begin{align}
    (r^{n-1}\G_{m}^{'} e^{\int^{r}_1 \B_m dt})'=0,
\end{align}
where $(\cdot)'$ denotes the derivative with respect to $r$. We have,

\begin{align}
r^{n-1}\G_{m}'(r)e^{\int^{r}_{1}\B_m  dt}=a^{n-1}\G'_{m}(a)e^{\int^{a}_{1}\B_m  dt},
\end{align}
and thus, 
\begin{align}\label{eq13}
    \G'_{m}(r)=\frac{1}{r^{n-1}}a^{n-1}\G'_{m}(a)e^{-\int_{a}^{r} \B_m dt}
\end{align}
For a fixed $m$, we distinguish two cases, 
\begin{itemize}

\item $0<a< r\leq 1- \frac{1}{m}$.

In this case, the integral on the exponential gives us, 
\begin{align}
   -\int_{a}^{r} \B_m dt= \int_{a}^{r} \frac{C}{1-t} dt = C log(\frac{1-a}{1-r})
\end{align}
Thus, we have,
\begin{align}\label{eqone}
    \G'_{m}(r)=\frac{1}{r^{n-1}}a^{n-1}\G'_{m}(a)\Big(\frac{1-a}{1-r}\Big)^C.
\end{align}

\item $0<a<1-\frac{1}{m}<r<1$.
In this case, the integral on the exponential gives,
\begin{multline}\label{eqtwo}
    -\int_{a}^{r} \B_m dt= \int_{a}^{1- \frac{1}{m}} \frac{C}{1-t} dt  +  \int_{1-\frac{1}{m}}^{r} \frac{C}{1-(1- \frac{1}{m})} dt = \int_{a}^{1-\frac{1}{m}} \frac{C}{1-t} dt  +  (r -1+\frac{1}{m}) Cm\\ = C\log \big(m(1-a)\big) + Cm(r -1+\frac{1}{m})
\end{multline}
\end{itemize}

Note that $\G_m(1)=0$.

\bigskip
We note, using the maximum principle, that the Green's function is radially non-increasing, and thus the radial derivative is non positive. We now note the existence of some positive constant $C_0$, independent of $m$, so that there exists a decreasing sequence\footnote{Instead of a sequence, it is enough to just find a single point with the stated property.} $a_{k,m}\to 0, k\geq 0 ,$ and $ a_{0,m} <1/4$, so that for each $k\geq 0$
\begin{align}\label{lowerbound}
    \big|\frac{d \G_m}{d r}\Big|_{a_{k,m}}\big|\geq  \frac{C_0}{a_{k,m}^{n-1}}.
\end{align}
(Here, as usual, by $a_{k,m}$ we mean the point $(a_{k,m},0,\dots,0)$.)
If this is not true, and for any $\theta>0$, for all sufficiently small $a\leq a_{0,m}$, we have
\begin{align}\label{eqq42}
\big|\frac{d \G_m}{d r}\big|_{a}<\frac{\theta}{a^{n-1}},
\end{align} 
then we will get a contradiction to the lower bound coming from \cref{thm1}, since in that case by integrating \cref{eqq42}, we will get, noting again that $\frac{\partial \G}{\partial r}<0$,
\begin{align}\label{lowerr}
    \G_m(p,0)-\G_m(a_{i,m},0)\leq  \frac{\theta}{n-2} \Big( \frac{1}{p^{n-2}}-\frac{1}{a_{i}^{n-2}}\Big).
\end{align}

       Thus, when $\theta$ is sufficiently small compared to  $K$ we get a contradiction to \cref{lower} , as  \cref{lowerr} shows that the increase of the values of the Green's function is slow enough as $p\to 0$, 
\begin{align}
        \G_m(p,0)\leq \G_m(a_{i,m},0)+ \frac{\theta}{n-2} \Big( \frac{1}{p^{n-2}}-\frac{1}{a_{i}^{n-2}}\Big).    
\end{align}

       so there would have to exist some point $p\leq a_{0,m}$ so that $\G_m(p,0)<\frac{K(M)}{p^{n-2}}$, which is a contradiction to \cref{thm1}.

For simplicity, consider the special case of $C=1$.
For a fixed $m$, choosing any $a_{k,m}$, with the above property, we get, $|(a_{k,m})^{n-1}\G'_{m}(a_{k,m})|>C_0$, independent of $m$. Now integrating \cref{eq13}, from $1/2$ to $1$, we get, for any $m$, using \cref{eqtwo},
\begin{align}\label{imp}
   \G_m(1)- \G_m (\frac{1}{2})=-\G_m (\frac{1}{2})=(a^{n-1}_{k,m} \G'_m (a_{k,m})) \Bigg(\int_{\frac{1}{2}}^{1- \frac{1}{m}} \frac{1}{r^{n-1}} \Big(\frac{1-a_{k,m}}{1-r}\Big) dr + \int_{1- \frac{1}{m}}^{1} \frac{1}{r^{n-1}} e^{-\int_{a}^{r} \B_m dt} dr \Bigg)
\end{align}
So, we have, noting again that $\G'_{m}(a_{k,m})<0$,
\begin{align}
     \G_m(\frac{1}{2})\geq C_0 \Bigg(\int_{\frac{1}{2}}^{1- \frac{1}{m}} \frac{1}{r^{n-1}} \Big(\frac{1-a_{k,m}}{1-r}\Big) dr + \int_{1- \frac{1}{m}}^{1} \frac{1}{r^{n-1}} e^{-\int_{a}^{r} \B_m dt} dr \Bigg)\geq C_0  \Bigg(\int_{\frac{1}{2}}^{1- \frac{1}{m}} \frac{1}{2r^{n-1}} \Big(\frac{1}{1-r}\Big) dr  \Bigg).
\end{align}
The above inequality holds with the bound $C_0$ independent of $m$. As $m\to \infty$, clearly the right hand side goes to infinity. 

Thus, we can't get pointwise upper bounds for $\G_m$, independent of $m$. 
\end{proof}
\begin{remark}
    We note that if we took the constant $C<1$ in \cref{two}, then for $G_{m}(x,0)$, one does get a finite value for $G_m(\frac{1}{2})$. More generally, in \cite{Pat25a}, it is shown that for any not necessarily bounded domain $\Omega$, there exists a sufficiently small $\eps(\Omega)$ so that when the constant of the limiting drift of  \cref{imp2}, is taken as $M=\eps(\Omega)$, one gets pointwise upper bounds as expected, in balls bounded away from the pole of the Green's function, and the constants are uniform over the domain. \footnote{In fact, we show more generally the same result in \cite{Pat25a}, when the drift is 'small' on average in Whitney balls.}
\end{remark}
\begin{remark}
    This also immediately gives us that the solution for the Poisson-Dirichlet problem is not well defined in the limit as $m\to \infty$; consider the data $f=1_{B(\frac{1}{2}, \eps)}$, a ball of radius $\eps$ for a small enough $\eps$. Consider the problem $L_m u_m=f$ in $B(0,1)$, $u=0 \ \text{on}  \ \partial B$. In this case, using Harnack's inequality and Theorem 2, we have $u(0)=\int G_m (0,y)f(y)dy\to \infty$ as $m\to \infty$. Further, using Harnack's inequality, we also get that $u_m(t)\to \infty$ as $m\to \infty$ for any $t\in B(0,\frac{1}{4})$ , and thus one doesn't get any solution for the limiting drift, in $W^{1,2}(B(0,1))$.
\end{remark}
\begin{remark} From the symmetry of the problem, and the definition of the elliptic measure with pole at the origin, it is clear that for any fixed $m$, the elliptic measure is equivalent to the Lebesgue measure on the boundary, and thus still doubling. In other words, identifying the boundary with $\mathbb{S}^{n-1}$, if we consider any surface ball $\Delta(p,r)=B(p,r)\cap \mathbb{S}^{n-1}$ , for some $p\in \mathbb{S}^{n-1}$, and consider the boundary function 
\begin{align}
    f(x)=\begin{cases}
        1,& x\in B\\
        0,& x\in S^n \setminus B,
    \end{cases}
\end{align} 
then for each $m$, the elliptic measure $\omega^{0}_{L_m}(\Delta(p,r))$ corresponding to the operator $L_m$  with pole at the origin, gives us the relative surface measure of $\Delta(p,r)$ in $\mathbb{S}^{n-1}$.


\end{remark}

\section{Acknowledgment.}
The author acknowledges useful feedback from Steve Hofmann.

\bigskip

Ethical Approval : Not applicable.

Funding : Not applicable.

Availability of data and materials: Not applicable

\end{document}